\documentclass[11pt]{amsart}

\usepackage[colorlinks]{hyperref}
\usepackage{amssymb}
\usepackage{mathrsfs}
\usepackage[all,cmtip]{xy}
\usepackage{cite}
\usepackage{color}
\usepackage{tikz-cd}
\usepackage{framed}

\newtheorem{theorem}{Theorem}[section]
\newtheorem{corollary}[theorem]{Corollary}
\newtheorem{lemma}[theorem]{Lemma}
\newtheorem{proposition}[theorem]{Proposition}

\newtheorem*{Theorem DA}{Theorem A}
\newtheorem*{Theorem LRR}{Theorem B}
\newtheorem*{Theorem Pel}{Theorem C}
\newtheorem*{Theorem Al}{Aleksandrov's theorem}

\theoremstyle{definition}
\newtheorem{definition}[theorem]{Definition}
\numberwithin{equation}{section}

\newcommand{\D}{\mathbb D}
\newcommand{\R}{\mathbb R}
\newcommand{\C}{\mathbb C}

\newcommand{\T}{\mathbb T}
\newcommand{\E}{\mathbb E}
\newcommand{\cal}{\mathcal}
\newcommand{\ol}{\overline}
\newcommand{\Ran}{{\rm Ran}\,}
\newcommand{\Ker}{{\rm Ker}\,}
\newcommand{\la}{\langle}
\newcommand{\ra}{\rangle}

\newcommand{\one}{{\bf 1}}
\newcommand{\M}{\mathcal M}
\newcommand{\X}{\mathfrak{X}}

\newcommand{\sgn}{{\rm sgn}\,}

\allowdisplaybreaks

\title{Extension of contractive projections}

\author{Xiangdi Fu}
\address{Xiangdi Fu: School of Fundamental Physics and Mathematical Sciences, HIAS, University of Chinese Academy of Sciences, Hangzhou, 310024, China}
\email{xdfu@ucas.ac.cn}

\author{Kunyu Guo}
\address{Kunyu Guo: School of Mathematical Sciences, Fudan University, Shanghai, 200433, China}
\email{kyguo@fudan.edu.cn}

\author{Dilong Li}
\address{Dilong Li: School of Mathematical Sciences, Fudan University, Shanghai, 200433, China}
\email{dlli23@m.fudan.edu.cn}


\makeatletter
\@namedef{subjclassname@2020}{\textup{2020} Mathematics Subject Classification}
\makeatother

\begin{document}
	\begin{abstract}
		Through the establishment of several extension theorems, we provide explicit expressions for all contractive projections and 1-complemented subspaces in the Hardy space $H^p(\T)$ for $1\leq p<\infty$, $p\neq 2$. Our characterization leads to two corollaries: first, all nontrivial 1-complemented subspaces of $H^p(\T)$ are isometric to $H^p(\T)$; second, all contractive projections on $H^p(\T)$ are restrictions of contractive projections on $L^p(\T)$ that leave $H^p(\T)$ invariant. The first corollary provides examples of prime Banach spaces \emph{in the isometric sense}, while the second answers a question posed by P. Wojtaszczyk in 2003 \cite{Wo03}.

	\end{abstract}
	\maketitle
	\tableofcontents

	\section{Introduction}
	
	\subsection{Background}
	
	Let $X$ be a Banach space. A linear operator $P: X\to X$ is called a {\it projection} if $P^2=P$, and we always assume $P\neq 0$. A closed subspace $M$ of $X$ is said to be {\it complemented} if there exists a bounded projection $P$ on $X$, such that the range of $P$ is $M$. Clearly, if $P$ is a nonzero projection, then $\|P\|\geq 1$. These projections satisfying $\|P\|=1$, called {\it contractive projections} (alternatively called {\it norm-one projections}), and their ranges, called {\it 1-complemented subspaces}, are of particular interest. By the Hahn-Banach theorem, every finite-dimensional subspace is complemented and every 1-dimensional subspace is 1-complemented. Therefore, by \emph{nontrivial complemented subspaces} we refer to complemented subspaces of infinite dimension, while by \emph{nontrivial 1-complemented subspaces} we refer to 1-complemented subspaces of dimension greater than one.
	
	Contractive projections in Banach spaces have received considerable attention from functional analysts since the field's inception, partly because they are a natural generalization of orthogonal projections in Hilbert spaces \cite{Ka39,Bo42,Ja45}. Moreover, 1-complemented subspaces also serve as important isometric invariants in the classification theory of Banach spaces \cite{LL71,Ros03}. In general, it seems hopeless to completely determine the contractive projections and 1-complemented subspaces of any given Banach space. Indeed, even for certain classical spaces, the known results are far from trivial. In his early work, A. Grothendieck \cite{Gr55} proved that $1$-complemented subspaces in any $L^1$-space are isometric to an $L^1$-space. For a probability measure $\mu$, R. G. Douglas \cite{Do65} completely characterized contractive projections on $L^1(\mu)$, and T. And\^o \cite{An} extended Douglas's results to $L^p(\mu)$ for $1<p<\infty$, $p\neq 2$. As indicated in their work, to characterize contractive projections on $L^p(\mu)$, it suffices to consider those that preserve the constants, since any contractive projection on $L^p(\mu)$ can be transformed into a contractive projection of this type. The theorems of Douglas and And\^o assert that if $P$ is a contractive projection on $L^p(\mu)$ $(1\leq p<\infty, \,p\neq 2)$ satisfying $P\one =\one$, then $P$ is a conditional expectation. As a corollary, every $1$-complemented subspace of $L^p(\mu)$ with $\mu$ a probability measure is isometric to an $L^p$-space. Later, L. Tzafriri \cite{Tz} proved that this statement remains true for any positive measure $\mu$, thereby extending Grothendieck's theorem to all $1\leq p< \infty$, $p\neq 2$. These seminal results in $L^p$-spaces have profoundly influenced further research on contractive projections and 1-complemented subspaces in various Banach spaces; see \cite{DHP90,FR,Ray,AF,Do95,Ran,LRG07} and the references therein. The excellent surveys \cite{Ran,Pel79} offer valuable historical insights and a detailed bibliography on these topics.

	However, despite the extensive work of many authors on contractive projections and 1-complemented subspaces, very little is known about them in spaces of analytic functions. The first result in this direction, to the best of the authors' knowledge, is due to P. Wojtaszczyk in 1979 \cite{Wo79}, who proved that $H^\infty(\T)$ and the disc algebra admit a few $1$-complemented subspaces of finite dimension; see \cite[Theorem 4.3]{Wo79}. For $1\leq p<\infty,p\neq 2$, F. Lancien, B. Randrianantoanina, and E. Ricard \cite[Theorem 3.1]{LRR} showed that $H^p(\T)$ does not admit any nontrivial 1-complemented subspace of finite dimension, and they also constructed several interesting 1-complemented subspaces of $H^p(\T)$. Nevertheless, the general structure of contractive projections on $H^p(\T)$ remains unknown, as Wojtaszczyk remarked in his lecture notes \cite{Wo03}:
	
	{\it ------``Very little is known about norm one projections in $H^1$ spaces. All known such projections are restrictions of norm one projections on $L^1$ which leave $H^1$ invariant. It is not known if there are any others.''}
	
	In addition to the above progress in $H^p(\T)$, O. F. Brevig, J. Ortega-Cerd\`a, and K. Seip \cite{BOS} recently obtained some remarkable results in $H^p(\T^d)$, the Hardy space over the $d$-dimensional torus. They completely determined those nonempty sets $\Gamma \subseteq \mathbb N^d_0$ for which the corresponding  idempotent Fourier multiplier
	\begin{align*} 
		P_\Gamma: \sum_{\alpha \in \mathbb N_0^d} c_\alpha z^\alpha \mapsto  \sum_{\alpha \in \Gamma} c_\alpha z^\alpha
	\end{align*}
	extends to a contraction on $H^p(\T^d)$ where $1\leq p\leq \infty$ and $d\geq 1$. Their results provide a full description of {\it contractive projections induced by idempotent Fourier multipliers} on $H^p(\T^d)$.

	\subsection{Main Results}
	
	Our main aim in this paper is to present explicit expressions for all contractive projections and 1-complemented subspaces in the Hardy space $H^p(\T)$ with $1\leq p<\infty$, $p\neq 2$.
	
	For \(1\leq p \leq \infty\), the Hardy space on the unit circle \(H^p(\T)\) consists of all functions in \(L^p(\T)\) whose Fourier coefficients vanish for the negative indices, that is, 
	\[
	H^p(\T):= \left\{ f\in L^p(\T): \frac{1}{2\pi}\int_{-\pi}^{\pi} f(e^{i\theta}) e^{i n \theta} d\theta =0,\quad \forall n\in \mathbb N \right\}.
	\] 
	Our terminology in the theory of Hardy spaces is consistent with the classical references \cite{Ho62,Gar81}.
	
	To simplify the statement of our main theorem, we will henceforth write $H^p$ for $H^p(\T)$, $L^p$ for $L^p(\T)$, and introduce the following notations. 
	
	\begin{definition}\label{1.1}
		For $1\leq p<\infty$, $p\neq 2$, we define the set $\X_p$ to consist of all pairs of functions $(\eta, \varphi)$ that satisfy the following conditions: 
		\begin{enumerate} 
			\item $\eta$ is an inner function with $\eta(0)=0$.
			\item $\varphi\in H^p$, and $\|\varphi\|_p=1$.
			\item Let $\varphi=\xi\cdot F$ be the inner-outer factorization of $\varphi$. The triple $(\xi, F, \eta)$ satisfies $$\xi\cdot F^{p/2}\in H^2\ominus \eta H^2.$$
		\end{enumerate}
		For any $(\eta,\varphi)\in\X_p$, the space $$H^p_{\eta,\varphi}:= \{\varphi\cdot f\circ\eta: f\in H^p\}$$ is a closed subspace of $H^p$ that is isometric to $H^p$; see Proposition \ref{4.2}. Moreover, we denote by $dm$ the normalized Lebesgue measure on $\T$, and by $\E_\varphi(\cdot |\eta)$ the conditional expectation operator on the weighted space $L^p(|\varphi|^pdm)$ with respect to the $\sigma$-algebra generated by $\eta$.
	\end{definition}

	With these notations, our main theorem reads as follows:
	
	\begin{theorem}\label{1.2}
		Let $1\leq p< \infty$, $p\neq 2$. Then $P: H^p\to H^p$ is a contractive projection of rank greater than one if and only if $$P=\varphi\E_{\varphi}(\cdot /\varphi| \eta)\big|_{H^p}$$ for some $(\eta,\varphi)\in \X_p$.
		Additionally, $M$ is a nontrivial 1-complemented subspace of $H^p$ if and only if $$M=H^p_{\eta,\varphi}$$ for some $(\eta,\varphi)\in \X_p$.
	\end{theorem}

	\noindent In particular, if we set $$\widetilde{P}:=\varphi \E_\varphi(\cdot/\varphi |\eta),$$ then $\widetilde{P}$ is a contractive projection on $L^p$ satisfying $$\widetilde{P}\big|_{H^p}=P.$$ This answers the problem posed by Wojtaszczyk in \cite{Wo03}, as mentioned earlier. Actually, we will show that such an extension $\widetilde{P}$ is unique. 
	
	\begin{corollary}\label{1.3}
		Let $1 \leq p <\infty$, $p\neq 2$. For any contractive projection $P:H^p\to H^p$, there exists a unique contractive projection $\widetilde{P}$ on $L^p$ such that $P$ is the restriction of $\widetilde{P}$ on $H^p$.
	\end{corollary} 
	
	Moreover, since all spaces $H^p_{\eta, \varphi}$ are isometric to $H^p$, Theorem \ref{1.2} also implies the following corollary.
	
	\begin{corollary}\label{1.4}
		Let $1\leq p< \infty$, $p\neq 2$. Every nontrivial 1-complemented subspace of $H^p$ is isometric to $H^p$.
	\end{corollary}
	
	\noindent This corollary is reminiscent of a celebrated theorem of A. Pe{\l}czy\'nski \cite{Pel60}, which states that the sequence spaces $\ell^p$ $(1\leq p<\infty)$ are prime (an infinite-dimensional Banach space $X$ is called {\it prime} if every nontrivial complemented subspace of $X$ is isomorphic to $X$). We note that Corollary \ref{1.4} directly implies the theorem of Lancien, Randrianantoanina, and Ricard \cite[Theorem 3.1]{LRR}. Theorem \ref{1.2} also leads to the following codimensional version of \cite[Theorem 3.1]{LRR}.
	
	
	\begin{corollary}\label{1.5}
		Let $1\leq p< \infty$, $p\neq 2$. Then $H^p$ contains no proper $1$-complemented subspace of finite codimension.
	\end{corollary}

	The sufficiency part of Theorem \ref{1.2} is relatively easy and will be presented at the beginning of Section 4. The necessity part is more involved and will be organized into three steps, corresponding to three subsections of Section 4. 
	
	To prove the necessity, our starting point is to establish the following two extension theorems for contractive projections on subspaces of $L^p(\mu)=L^p(\Omega, {\bf \Sigma}, \mu)$, where $(\Omega, {\bf \Sigma}, \mu)$ is a complete probability space. These extension theorems are of independent interest.

	\begin{theorem}\label{1.6}
		Suppose that $1\leq p<\infty$, $p\notin 2\mathbb N$, and that $X$ is a closed subspace of $L^p(\mu)$ containing the constant functions. Let $P$ be a contractive projection on $X$ satisfying $P\one =\one$, then $$Pf=\E(f| {\bf\Sigma'}), \quad \forall f\in X,$$ where ${\bf \Sigma'}$ is the sub $\sigma$-algebra generated by functions in $\Ran P$.
	\end{theorem} 
	
	For $p=2k+2$ with $k\in \mathbb N$, Theorem \ref{1.6} fails in general, but remains true if the multiplier algebra of $X$ enjoys certain additional conditions. For a closed subspace $X$ of $L^p(\mu)$ containing the constants, we define the multiplier algebra of $X$ by $$\cal M(X):= \{h\in X: hf\in X, \forall f\in X\}.$$ 
	It is routine to verify that $\cal M(X)$ is a closed subalgebra of $L^\infty(\mu)$.

	\begin{theorem}\label{1.7}
		Suppose $p=2k+2$ with $k\geq 1$. Let $X$ be a closed subspace of $L^p(\mu)$ containing the constants. Suppose that $\cal M(X)$ satisfies one of the following conditions:
		\begin{enumerate}
			\item The multiplier algebra $\cal M(X)$ is dense in $X$.
			\item The intersection of $\cal M(X)+\ol{\cal M(X)}$ and the closed unit ball of $L^\infty(\mu)$ is $w^*$-dense in the closed unit ball of $L^\infty(\mu)$.
		\end{enumerate}
		Then any contractive projection $P$ on $X$ satisfying $P\one =\one$ is of the form $$Pf=\E(f| {\bf\Sigma'}), \quad \forall f\in X,$$ where ${\bf \Sigma'}$ is the sub $\sigma$-algebra generated by functions in $\Ran P$.
	\end{theorem}
	
	The proofs of the extension theorems will be presented in Section 3, while some necessary preliminaries are introduced in Section 2. We point out that the proofs of the extension theorems do not rely on the results of Douglas and And\^o, and in fact provide a new proof of their theorems by taking $X=L^p(\mu)$. Moreover, from the extension theorems, we can easily derive a weak version of Theorem \ref{1.2} (see Theorem \ref{4.6}) without much use of properties specific to the Hardy spaces $H^p$. Interestingly, this weaker form is sufficient to answer Wojtaszczyk's question (Corollary \ref{1.3}), and to cover the Lancien-Randrianantoanina-Ricard theorem along with its codimensional version (Corollary \ref{1.5}); see Appendix.

	\vspace{0.3cm}
	
	\noindent{\bf Notation.} The letter $\mathbb{N}$ denotes the set of all positive integers, and we define $\mathbb{N}_0 := \mathbb{N} \cup \{0\}$. Throughout this paper, $(\Omega, \mathbf{\Sigma}, \mu)$ represents a complete probability space. For a family $\mathcal{G}$ of measurable functions on $(\Omega, \mathbf{\Sigma}, \mu)$, the $\sigma$-algebra generated by $\mathcal{G}$, denoted $\mathbf{\Sigma}(\mathcal{G})$, refers to the smallest complete sub $\sigma$-algebra such that every function in $\mathcal{G}$ is measurable. For a $\mathbf{\Sigma}$-measurable function $g$, we denote by $\mathbb{E}(\cdot|g)$ the conditional expectation with respect to the $\sigma$-algebra generated by the single function $g$.

	\section{Preliminaries}

	\subsection{Conditional Expectations} 
	The conditional expectation, introduced by A. Kolmogorov in 1933, is a fundamental concept in probability theory. Let $(\Omega,{\bf \Sigma}, \mu )$ be a complete probability space, and let ${\bf \Sigma}'$ be a complete sub $\sigma$-algebra of ${\bf \Sigma}$. For $f\in L^1(\Omega, {\bf \Sigma}, \mu)$, the conditional expectation of $f$ with respect to the sub $\sigma$-algebra ${\bf \Sigma}'$, denoted by $\E(f| {\bf \Sigma}')$, is the unique function in $L^1(\Omega, {\bf \Sigma}', \mu)$ satisfying 
	\begin{align*}
		\int_{F} \E(f| {\bf \Sigma}') d\mu = \int_{F} f d\mu, \quad  \forall F\in {\bf \Sigma}'.
	\end{align*}
	The existence of $\E(f|{\bf \Sigma}')$ is guaranteed by the Radon-Nikodym theorem. 
	Some basic properties of conditional expectations are as follows: 
	\begin{enumerate}
		\item 
		$\E(\cdot|{\bf \Sigma}')$ is a contractive projection from $L^p(\Omega, {\bf \Sigma}, \mu)$ onto $L^p(\Omega, {\bf \Sigma}', \mu)$ for $1\leq p\leq \infty$. When $p=2$, $\E(\cdot| {\bf \Sigma}')$ is the orthogonal projection.
		\item $\E(\cdot|{\bf \Sigma}')$ is positive, i.e. if $f\geq 0$, then $\E(f|{\bf \Sigma}') \geq 0$.
		\item $\E(\cdot|{\bf \Sigma}')$ satisfies the {\it averaging identity}: 
		$$\E\bigg(f \cdot \E(g| {\bf \Sigma}') \bigg|{\bf \Sigma}'\bigg)= \E(f|{\bf \Sigma}') \cdot \E(g| {\bf \Sigma}').$$
	\end{enumerate}
	We refer to \cite{Du} for further properties and applications of conditional expectations in probability theory, and to \cite{Ran05, DHP90} for comprehensive operator-theoretic characterizations of conditional expectation-type operators.

	\subsection{Orthogonality in $L^p$}
	Let $X$ be a Banach space, and let $Y$ be a closed subspace of $X$. An element $g \in X$ is said to be Birkhoff-James orthogonal to $Y$ if  
	$$
	\|g\| \leq \|f + g\|, \quad \forall f \in Y.
	$$
	This notion of orthogonality is closely related to contractive projections. It is straightforward to observe that a projection $P: X \to X$ is contractive if and only if every $g \in \Ran P$ is Birkhoff-James orthogonal to ${\rm Ker}\, P$.  
	
	We need certain criteria for orthogonality in $L^p$-spaces; see \cite[Threom 4.2.1, Theorem 4.2.2, and Exercise 4.2.3]{Sha}. Below, we reformulate these criteria in terms of contractive projections.

	\begin{lemma}\label{2.1}
		Suppose that $1< p<\infty$, and that $X$ is a closed subspace of $L^p(\mu)$. Let $P$ be a bounded projection on $X$, then $P$ is of contractive if and only if 
		\begin{align*}
			\int_{\Omega} |g|^{p-1}\sgn g \cdot \ol{(f-Pf)} d\mu =0, \quad \forall g\in \Ran P, \quad \forall f\in X.
		\end{align*}
	\end{lemma}
	
	\noindent Here, $\sgn : \mathbb C \to \T \cup \{0\}$ is the signum function defined as $$\sgn z =
	\begin{cases}
		z/ |z|, & z\neq 0;\\
		0, & z=0.
	\end{cases}$$ 
	\begin{lemma}\label{2.2}
		Suppose that $X$ is a closed subspace of $L^1(\mu)$. Let $P$ be a bounded projection on $X$, then $P$ is of contractive if and only if 
		\begin{align}\label{eq2.1}
			\bigg|\int_{\Omega} \sgn g \cdot \ol{(f-Pf)} d\mu\bigg| \leq \int_{Z_g}|f-Pf| d\mu, \quad \forall g\in \Ran P, \quad \forall f\in X,
		\end{align}
		where $Z_g=\{x\in \Omega: g(x)=0 \}$. In particular, if $\mu(Z_g)=0$, then \eqref{eq2.1} results in $$\int_{\Omega} \sgn g \cdot \ol{(f-Pf)} d\mu=0, \quad \forall f\in X.$$
	\end{lemma}

	\subsection{A Uniqueness Lemma}
	
	We denote by $dA$ the area measure on $\C$. For an integrable function $f$ on $\C$, we denote the Fourier transform of $f$ by $$\mathcal F(f)(w)=\int_{\C} f(z)e^{-2\pi i {\rm Re}\,(w\ol{z})} dA(z), \quad w\in \C,$$ 
	and the inverse Fourier transform of $f$ by
	$$
	\mathcal F^{-1}(f)(w)=\int_{\C} f(z)e^{2\pi i {\rm Re}\,(w\ol{z})} dA(z), \quad w\in \C.
	$$
	It is evident that if $f$ is treated as a function on $\R^2$ by identifying $\C$ with $\R^2$, then the above notation agrees with the ordinary Fourier transform and the ordinary inverse Fourier transform. 
	
	The following lemma is the first key step in the proof of Theorem \ref{1.6}.
	
	\begin{lemma}\label{2.3}
		Suppose $1\leq p< \infty $, $p\notin 2\mathbb N$. If $(1+|z|^{p-1})\varphi (z)\in L^1(\C)$ satisfying 
		\begin{equation*}
			|z|^{p-1}\sgn z * \varphi=0
		\end{equation*}
		almost everywhere, then $\varphi=0$ almost everywhere.
	\end{lemma}
	
	\begin{proof}
		By the assumption, we have
		\begin{equation*}
			\mathcal F^{-1}(f)*|z|^{p-1}\sgn z * \varphi=0, \quad\forall f\in C_c(\C\setminus\{0\}).
		\end{equation*}
		Note that for $p\geq 1$ and $p\notin 2\mathbb N$, the Fourier transform of $|z|^{p-1}\sgn z$, in the sense of distribution, is the generalized function $|z|^{-p-1}\sgn z$ multiplied by a constant $\gamma_p$; see \cite[pp. 383, Equation (9)]{GS64}. Then
		$$
		\mathcal F(\mathcal F^{-1}(f)*|z|^{p-1}\sgn z )=\gamma_p f\cdot |z|^{-p-1}\sgn z,\quad \forall f\in C_c(\C\setminus\{0\}),
		$$
		and so
		$$
		f\cdot |z|^{-p-1}\sgn z\cdot\mathcal F(\varphi)=0,\quad \forall f\in C_c(\C\setminus\{0\}).
		$$
		It follows that $\mathcal F(\varphi)=0$ on $\C\setminus\{0\}$ and hence $\varphi=0$ almost everywhere.
	\end{proof}
	
	Lemma \ref{2.3} fails when $p\in 2\mathbb N$. Indeed, suppose $p=2k+2$ with $k\in \mathbb N_0$. Then the equation $|z|^{p-1}\sgn z \ast \varphi =0$ is equivalent to 
	\begin{align*}
		\int_{\C} z^n \ol{z}^{m} \varphi(z) dA(z) = 0 ,\quad 0\leq n\leq k+1, 0\leq m \leq k,
	\end{align*}
	which admits infinitely many nontrivial solutions; for example, $\varphi(z)=z^{\ell }e^{-|z|^2}$ with $\ell \geq k+1$. Actually, in the next section we shall see that the failure of Lemma \ref{2.3} for even $p$ is the main reason why Theorem \ref{1.7} does not hold in general for these values of $p$. It is worth noting that this phenomenon also appears in the Rudin-Plotkin Equimeasurability Theorem \cite{Rud76}.

	\section{The Extension Theorems}
	
	In this section, we present the proofs of the extension theorems. The proof of Theorem \ref{1.6} is given in the first subsection, while the proof of Theorem \ref{1.7} is provided in the second.

	\subsection{The Case $p\geq 1$ and $p \notin 2\mathbb N$}
	
	For convenience, we restate Theorem \ref{1.6} below.
	
	\begin{theorem}\label{3.1}
		Suppose that $1\leq p<\infty$, $p\notin 2\mathbb N$, and that $X$ is a closed subspace of $L^p(\mu)$ containing the constant functions. Let $P$ be a contractive projection on $X$ satisfying $P\one =\one$, then $$Pf=\E(f| {\bf\Sigma'}), \quad \forall f\in X,$$ where ${\bf \Sigma'}$ is the sub $\sigma$-algebra generated by functions in $\Ran P$.
	\end{theorem}
	
	The proof of Theorem \ref{3.1} will be completed in two steps.
	Note that $Pf$ is ${\bf\Sigma'}$-measurable, the desired identity in Theorem \ref{3.1} can be rewritten as
	\begin{align}\label{eq3.1}
		\E(f-Pf| {\bf\Sigma'})=0,\quad  \forall f\in X.
	\end{align} 
	The first step is to establish the following local version of \eqref{eq3.1}:
	\begin{align}\label{eq3.2}
		\E(f-Pf| g)=0,\quad  \forall f\in X, \quad \forall g\in \Ran P.
	\end{align}
	Indeed, we will prove \eqref{eq3.2} by applying the orthogonality lemma and the uniqueness lemma in Section 2.
	
	The next key step is to show that \eqref{eq3.2} implies \eqref{eq3.1}. Note that this implication is not straightforward, as the $\sigma$-algebra generated by functions in $\Ran P$ is generally distinct from the union of the $\sigma$-algebras generated by each individual element $g\in \Ran P$. In fact, we arrive at the following result, which is of independent interest. 
	
	\begin{lemma}\label{3.2}
		Let $Y$ be a linear subspace of ${\bf \Sigma}$-measurable functions. Suppose that $f
		\in L^1(\mu)$ and $\E(f| g)=0$ for each $g\in Y$, then $\E\big(f|{\bf \Sigma}(Y)\big)=0.$
	\end{lemma}
	
	Let us begin by proving Theorem \ref{3.1} on the assumption that Lemma \ref{3.2} holds, and then proceed to prove Lemma \ref{3.2}.
	
	\begin{proof}[Proof of Theorem \ref{3.1}]	
		We first consider the case $p=1$. Suppose $g\in \Ran P$, $f\in X$ are fixed. Since $\one \in \Ran P$, for any $w\in \C$, the function $w-g\in \Ran P$. Applying Lemma \ref{2.2} we obtain 
		$$\int_\Omega \sgn (w-g)\ol{(f-Pf)}d\mu =0, \quad \forall w\in \C\setminus A,$$ where $A=\{w\in \C: \mu(Z_{w-g})\neq 0\}$. The set $A$ is at most countable, because there are at most countably many values of $w\in \C$ such that the set $\{x: g(x)=w\}$ is not null.		
		Let $\sigma$ be the push-forward measure defined for Borel sets $E\subseteq \C$ as $$\sigma(E)=\int_{g^{-1}(E)}(f-Pf)d\mu.$$ Then $\sigma$ is a finite complex Borel measure on $\C$, and by the change of variable formula it holds that 
		$$\int_\C \sgn(w-z)d\sigma(z)=0,\quad \forall w\in\C\setminus A.$$ In particular, the convolution $$\sgn z * \sigma=0$$ almost everywhere. For any function $h$ in the Schwartz calss $\cal S(\C)$, the function $\varphi=\sigma*h\in L^1(\C)$, and $$\sgn z*\varphi=\sgn z * (\sigma*h)=(\sgn z*\sigma)*h=0.$$ It follows from Lemma \ref{2.3} that $\varphi=0$. The arbitrariness of $h\in \cal S(\C)$ implies the measure $\sigma$ is zero. Thus
		\begin{align*}
			\int_{g^{-1}(E)}(f-Pf)d\mu =0 
		\end{align*}
		for any Borel subset $E\subseteq \C$, which means
		\begin{align*}
			\E(f-Pf| g)=0.
		\end{align*}
		The above identity holds for any function $g\in \Ran P$, by applying Lemma \ref{3.2} we obtain
		\begin{align*}
			\E(f-Pf| {\bf \Sigma'})=0.
		\end{align*}
		Since $Pf$ is ${\bf \Sigma'}$-measurable, we conclude that 
		\[
		Pf = \mathbb{E}(f | {\bf \Sigma'}), \quad \forall f \in X.
		\]
		This completes the proof for the case $p = 1$. For $1 < p < \infty$ with $p \notin 2\mathbb{N}$, the proof follows the same reasoning as in the case $p = 1$, with Lemma \ref{2.1} replacing Lemma \ref{2.2}. The argument is even simpler, as the identity
		\[
		\int_\Omega |w-g|^{p-1} \sgn (w - g) \overline{(f - Pf)} \, d\mu = 0
		\]
		holds for all $w \in \mathbb{C}$.
	\end{proof}

	Finally, we prove Lemma \ref{3.2}.
	\begin{proof}[Proof of Lemma \ref{3.2}]
		Without loss of generality, we can assume that $f$ is real-valued, and that $Y$ is a real linear subspace of real-valued functions. Otherwise, we replace $Y$ by
		\[
		{\rm Re}\, Y := \{ {\rm Re}\, g : g \in Y \},
		\]
		so that ${\bf \Sigma}(Y) = {\bf \Sigma}({\rm Re}\, Y)$.
		
		Note that
		$${\bf \Sigma}(Y)=\bigcup_{\cal G \subseteq Y, \,\cal G \,\,{\rm is\,\,countable}} {\bf \Sigma}(\cal G).$$
		As a result, to prove $\E(f|{\bf \Sigma}(Y))=0$, it is sufficient to show $\E(f|\cal G)=0$ for any fixed countable subset $\cal G=\{g_1,g_2,g_3,\ldots\}$. For $n\geq 1$ we denote by $\cal G_n:=\{g_1,g_2,\ldots, g_n\}$. By Doob's martingale convergence theorem, we only need to check 
		$\E(f|\cal G_n)=0$ for any $n\in \mathbb N$. Define the map $$\Lambda:\Omega \to \R^n;\quad  \omega\mapsto \big(g_1(\omega), g_2(\omega),\ldots, g_n(\omega)\big).$$
		Let $v=\Lambda_*(fd\mu)$ be the pushforward measure on $\R^n$, defined for any Borel subset $E\subseteq \R^n$ as $$v(E)=\int_{\Lambda^{-1}(E)}fd\mu .$$
		Observe that any measurable set in ${\bf \Sigma}(\cal G_n)$ is of the form $\Lambda^{-1}(E)\Delta F$, where $E\subseteq\R^n$ is a Borel subset and $F$ is a null set. Therefore, to prove $\E(f|\cal G_n)=0$, it is sufficient to show the measure $v$ is zero. 
		
		For any $\alpha=(\alpha_1, \alpha_2,\ldots ,\alpha_n)\in \R^n,$ by the change of variable formula we have 
		\begin{align}\label{eq3.3}
			\int_{\R^n} e^{-2\pi i\alpha\cdot x} dv(x)= \int_{\Omega} e^{-2\pi i\alpha \cdot \Lambda(\omega)} f(\omega)d \mu(\omega)= \E\bigg[f\cdot \exp\bigg(-2\pi i\sum_{j=1}^n \alpha_j g_j\bigg)\bigg].
		\end{align}
		The assumption $\E(f | \sum_{j=1}^n \alpha_j g_j) = 0$ implies that $\E(f \cdot G) = 0$ for every bounded function $G$ measurable with respect to the $\sigma$-algebra generated by $\sum_{j=1}^n \alpha_j g_j$. In particular, the expectation in \eqref{eq3.3} equals zero, and hence the Fourier transform of $v$ vanishes. This completes the proof of Lemma \ref{3.2}.
	\end{proof}

	As mentioned in the introduction, Theorem \ref{3.1} no longer holds if the condition $p\notin 2\mathbb N$ is omitted. An elementary example is presented below to illustrate this failure.
	
	\noindent {\bf Example.}
	Let $p=2k\in 2\mathbb N$ and let $X$ be the closed subspace of $L^{2k}(\mathbb{T})$ spanned by $\{1,z,z^{k+1}, z^{k+2},\ldots \}$. Consider the projection $P$ in $X$ defined by
	$$
	P: a+bz+z^{k+1} r(z) \mapsto a+bz, \quad a,b \in \C, r\in \C[z].
	$$
	It is routine to verify that $P$ extends to a contractive projection on $X$. However, $P$ is not the restriction of any conditional expectation, because the coordinate function $z$ lies in $\Ran P$ while $P\neq {\rm Id}_X$.

	\subsection{The Case $p\in 2\mathbb N$}
	In this subsection, we consider the case $p = 2k + 2$ with $k\in \mathbb N$. 
	
	\begin{theorem}\label{3.3}
		Suppose $p=2k+2$ with $k\geq 1$. Let $X$ be a closed subspace of $L^p(\mu)$ containing the constants. Suppose that $\cal M(X)$ satisfies one of the following conditions:
		\begin{enumerate}
			\item The multiplier algebra $\cal M(X)$ is dense in $X$.
			\item The intersection of $\cal M(X)+\ol{\cal M(X)}$ and the closed unit ball of $L^\infty(\mu)$ is $w^*$-dense in the closed unit ball of $L^\infty(\mu)$.
		\end{enumerate}
		Then any contractive projection $P$ on $X$ satisfying $P\one =\one$ is of the form $$Pf=\E(f| {\bf\Sigma'}), \quad \forall f\in X,$$ where ${\bf \Sigma'}$ is the sub $\sigma$-algebra generated by functions in $\Ran P$.
		
	\end{theorem}
	
	\begin{proof}
		For $w_1,\ldots, w_{2k+1}\in \C$, $g_1,\ldots, g_{2k+1} \in {\rm Ran}\, P$ and $f\in X$, it follows from Lemma \ref{2.1} that 
		\begin{align*}
			\int_{\Omega} \bigg|\sum_{j=1}^{2k+1} w_jg_j\bigg|^{2k} \bigg(\sum_{j=1}^{2k+1} w_jg_j\bigg)\cdot\overline{(f-Pf)}d\mu =0.
		\end{align*}
		The left integral can be expressed as a trigonometric polynomial about variables $w_j$. By considering the coefficient of the term $w_1\cdots w_{k+1}\ol{w_{k+2}}\cdots \ol{w_{2k+1}}$, we obtain
		\begin{align*}
			\int_{\Omega} g_1\cdots g_{k+1} \ol{g_{k+2}}\cdots \ol{g_{2k+1}} \cdot\overline{(f-Pf)} d\mu =0.
		\end{align*}
		Since $\one\in X$ and $k\geq 1$, we have 
		\begin{align}\label{eq3.4}
			\int_{\Omega} g_1 g_2\ol{g_3}\cdot  \ol{(f-Pf)} d\mu =0,\quad \forall g_1,g_2,g_3\in \Ran P, \quad 	\forall f\in X,
		\end{align}
		and
		\begin{align}\label{eq3.5}
			\int_{\Omega} g\cdot \ol{(f-Pf)}\,d\mu =0. \quad \forall g\in \Ran P, \quad \forall f\in X.
		\end{align}
		Note that \eqref{eq3.5} means $\Ran P$ and $\Ker P$ is orthogonal in $L^2(\mu)$.
		
		We first claim that $P$ satisfies the averaging identity:
		\begin{align}\label{eq3.6}
			P(hPf)=Ph\cdot Pf, \quad \forall h\in \cal M(X), \quad \forall f\in X.
		\end{align}
		Observe that for any $k\in X$, it holds that
		\begin{align*}
			\int_\Omega P(hPf )\cdot \ol{k} \,d\mu\stackrel{\eqref{eq3.5}}{=}&\int_\Omega P(hPf )\cdot \ol{Pk} \,d\mu\\
			\stackrel{\eqref{eq3.5}}{=}& \int_\Omega h\cdot Pf \cdot\ol{Pk}\, \,d\mu\\
			\stackrel{\eqref{eq3.4}}{=}&\int_\Omega Ph\cdot Pf \cdot \ol{Pk} \,d\mu \\
			\stackrel{\eqref{eq3.4}}{=}&\int_\Omega Ph\cdot Pf\cdot \ol{k} \,d\mu.
		\end{align*}
		Thus we obtain
		\begin{align}\label{eq3.7}
			\int_\Omega P(hPf )\cdot \ol{k} \,d\mu=\int_\Omega Ph\cdot Pf \cdot \ol{k} \,d\mu, \quad \forall h\in\cal M(X), \quad \forall f , k\in X.
		\end{align}
		In particular, if we put $k=P(hPf )$ in \eqref{eq3.7}, then 
		\begin{align*}
			\|P(hPf)\|^2_2&=\int_\Omega P(hPf) \cdot \ol{P(hPf)}d\mu\\
			&\stackrel{\eqref{eq3.5}}{=}\int_\Omega P(hPf)\cdot\ol{h Pf}d\mu\\
			&\stackrel{\eqref{eq3.7}}{=}\int_\Omega Ph \cdot Pf \cdot \ol{hPf}d\mu\\
			&\stackrel{\eqref{eq3.4}}{=}\int_\Omega Ph\cdot Pf \cdot \ol{Ph\cdot Pf}d\mu\\
			&=\|PhPf\|^2_2.
		\end{align*}
		Combining this with \eqref{eq3.7}, we conclude that $P(hPf)=PhPf$. This proves the identities in \eqref{eq3.6}.

		Next, we claim that 
		\begin{align}\label{eq3.8}
			Ph\in L^\infty(\mu), \quad \forall h\in \M(X);\quad  {\rm and}\quad  \|Ph\|_{\infty} \leq \|h\|_{\infty}.
		\end{align} The argument employs the standard ``power trick''. By applying \eqref{eq3.6}, an easy induction yields $$(Ph)^n=P\big(h (Ph)^{n-1}\big)\in \Ran P, \quad \forall n\in \mathbb N.$$ 
		This implies $\|(Ph)^n\|_p\leq \|(Ph)^{n-1}\|_p\|h\|_\infty$. Applying this inequality iteratively, we see $$\|(Ph)^n\|_p^{1/n} \leq \|h\|_{\infty}, \quad \forall n\in \mathbb N.$$ Taking the limit as $n$ tends to infinity, it follows that $Ph \in L^\infty(\mu)$ and $\|Ph\|_{\infty} \leq \|h\|_\infty$. This proves the claim.

		Now, suppose that $\cal M(X)$ is dense in $X$. From the above argument it follows that, for every $h\in \M(X)$ and $n\in\mathbb N$, 
		$$(Ph)^n \in \Ran \, P\cap L^\infty(\mu).$$
		Therefore, by \eqref{eq3.4} we obtain
		\begin{align*}
			\int_\Omega (Ph)^n\ol{(Ph)^m} (f-Pf) d\mu =0, \quad \forall m,n \in \mathbb N_0,\quad \forall f\in X.
		\end{align*}
		Since $Ph$ is bounded, the subspace spaned by $\{(Ph)^n\ol{(Ph)^m}: \forall m,n \in \mathbb N_0\}$ is dense in $L^1(\Omega, {\bf \Sigma}(Ph), \mu)$.
		Thus
		\begin{align}\label{eq3.9}
			\E(f-Pf| Ph)=0, \quad \forall h\in \cal M(X),  \quad \forall f\in X.
		\end{align}
		Applying Lemma \ref{3.2}, we deduce from \eqref{eq3.9} that $$\E(f-Pf| {\bf \Sigma''}), \quad \forall f\in X,$$ where $\bf\Sigma''$ is the sub $\sigma$-algebra generated by all functions in $P(\cal M(X))$. Since $\cal M(X)$ is dense in $X$, $P(\cal M(X))$ is dense in $\Ran P$. This implies any function $g\in \Ran P$ is ${\bf \Sigma''}$-measurable. Hence ${\bf \Sigma''}={\bf \Sigma'}$ and $$Pf=\E(f| {\bf \Sigma}'),\quad f\in X.$$ This proves the first part of Theorem \ref{3.3}.

		Next we assume that $\cal M(X)$ satisfies the second condition in Theorem \ref{3.3}. From \eqref{eq3.8}, $P:\cal M(X)\rightarrow L^\infty(\mu)$ is a contraction with $P\one=\one$. By some classical results concerning positive maps in $C^*$-algebras \cite[Corollary 2.8 and Proposition 2.12]{Pa01}, the map 
		$$
		\widetilde{P}: \M(X)+\overline{\M(X)}\rightarrow L^\infty; h+\overline{k}\mapsto Ph+\overline{Pk},
		$$
		is well-defined and contractive.
		For any $f\in X$ and $h,k\in \M(X)$, it follows from \eqref{eq3.4} that
		$$
		\int_\Omega Pf\cdot h\,d\mu=\int_\Omega Pf\cdot Ph \,d\mu=\int_\Omega f\cdot Ph \,d\mu,
		$$
		and
		$$
		\int_\Omega Pf\cdot \overline{k}\,d\mu=\int_\Omega Pf\cdot \overline{Pk}\,d\mu=\int_\Omega f\cdot \overline{Pk}\, d\mu.
		$$
		As a consequence,
		\begin{align*}
			\left| \int_{\Omega}Pf \cdot(h+\overline{k})d\mu\right|&= \left|\int_\Omega f \cdot \widetilde{P}(h+\overline{k})d\mu\right| \\
			&\leq \|f\|_1\cdot \|\widetilde{P}(h+\overline{k})\|_\infty\\
			&\leq \|f\|_1 \cdot \|h+\overline{k}\|_\infty.
		\end{align*}
		By the assumption, the set $\{h+\ol{k}: h,k \in \cal M(X), \|h+\ol{k}\|_\infty \leq 1\}$ is $w^*$-dense in the closed unit ball of $L^\infty(\mu)$, so we obtain from the duality that  $$\|Pf\|_1\leq \|f\|_1, \quad \forall f\in X.$$
		This means $P$ extends to a contractive projection on the closure in $L^1(\mu)$ of $X$. An application of Theorem \ref{3.1} completes the proof.
	\end{proof}

	\section{Contractive Projections  on $H^p$}
	
	In this section, we will prove Theorem \ref{1.2}, which characterizes all contractive projections and 1-complemented subspaces in $H^p$, $1 \leq p < \infty$, $p \neq 2$. We mention that, in the case $0<p<1$, $H^p$ contains considerably rare contractive
	projections in contrast with the present setting \cite{FGL25}.

	Firstly, we show that for any $(\eta,\varphi)\in \X_p$, the space $H^p_{\eta,\varphi}$ is 1-complemented in $H^p$. For convenience, we review the definitions of $\X_p$ and $H^p_{\eta,\varphi}$.
	
	\begin{definition}\label{4.1}
		For $1\leq p<\infty$, $p\neq 2$, we define the set $\X_p$ to consist of all pairs of functions $(\eta, \varphi)$ that satisfy the following conditions: 
		\begin{enumerate} 
			\item $\eta$ is an inner function with $\eta(0)=0$.
			\item $\varphi\in H^p$, and $\|\varphi\|_p=1$.
			\item Let $\varphi=\xi\cdot F$ be the inner-outer factorization of $\varphi$. The triple $(\eta, \xi, F)$ satisfies $$\xi\cdot F^{p/2}\in H^2\ominus \eta H^2.$$
		\end{enumerate}
	\end{definition} 
	
	The next proposition lists the basic properties of the pairs in $\X_p$.
	
	\begin{proposition}\label{4.2}
		Suppose $(\eta,\varphi)\in \X_p$, then
		\begin{enumerate} 
			\item $\{\eta^k: k\in \mathbb Z\}$ is an orthonormal system in $L^2(\T, |\varphi |^p dm)$.
			\item $\mathbb E(|\varphi|^p | \eta)=\one$.
			\item The map $$T_{\eta,\varphi}: f \mapsto \varphi \cdot f\circ \eta$$ is an isometry on $H^p$.
		\end{enumerate} 
	\end{proposition}
	
	\begin{proof}
		Suppose $\varphi=\xi \cdot F$ is the inner-outer factorization of $\varphi$. Since $\xi\cdot F^{p/2} \in H^2\ominus \eta H^2$, for $k\in \mathbb Z$ we have
		\begin{align}\label{eq4.1}
			\int_{\T}{\eta}^{k} \cdot |\varphi|^p dm=\la \eta^{k}
			\xi F^{p/2}, \xi F^{p/2}\ra=\begin{cases}
				1, & k=0,\\
				0, & k\neq 0.
			\end{cases}
		\end{align} 
		From \eqref{eq4.1}, $\{\eta^k: k\in \mathbb Z\}$ is an orthonormal basis of $L^2(\T,{\bf \Sigma}(\eta), |\varphi|^p dm)$. This proves (1) and (2). Let $q$ be an analytic polynomial. Since $|q|^p$ is a continuous function on $\T$, there exists a sequence of trigonometric polynomials $\{p_n\}$ that converges uniformly to $|q|^p$ on $\T$. In particular, $p_n\circ \eta \to |q\circ \eta|^p$ in $L^\infty$, and hence
		\begin{align*}
			\int_\T |\varphi \cdot q\circ \eta|^p dm =\lim_{n\to \infty} \int_{\T} |\varphi|^p \cdot p_n\circ \eta \,dm \stackrel{\eqref{4.1}}{=}\lim_{n\to \infty}\int_{\T} p_n \,dm=\int_\T |q|^p dm.
		\end{align*}
		This shows $T_{\eta,\varphi}$ extends to an isometry on $H^p$.
	\end{proof}

	\begin{definition}\label{4.3}
		For $(\eta,\varphi)\in \X_p$, we define $$H^p_{\eta,\varphi}:= \Ran T_{\eta,\varphi}= \{\varphi\cdot f\circ\eta: f\in H^p\}.$$ By Proposition \ref{4.2}, $H^p_{\eta,\varphi}$ is a closed subspace of $H^p$ that is isometric with $H^p$.
	\end{definition}
	
	Recall that $\E_\varphi(\cdot |\eta)$ denotes the conditional expectation operator on the weighted space $L^p(|\varphi|^pdm)$ with respect to the $\sigma$-algebra generated by $\eta$.
	
	\begin{theorem}[The sufficiency part of Theorem \ref{1.2}]\label{4.4}
		For any $(\eta,\varphi)\in \X_p$, the operator $$Q:=\varphi \E_{\varphi}(\cdot/\varphi |\eta)$$ is a contractive projection on $L^p$ that carries $H^p$ onto $H^p_{\eta,\varphi}$. In particular, the space $H^p_{\eta,\varphi}$ is $1$-complemented in $H^p$.
	\end{theorem}
	
	\begin{proof}
		It is evident that $Q$ is a contractive projection on $L^p$. To show $Q$ maps $H^p$ onto $H^p_{\eta,\varphi}$, it is sufficient to prove:

		\begin{enumerate}
			\item $Q$ maps $H^p$ into $H^p_{\eta,\varphi}$.
			\item For $f\in H^p_{\eta,\varphi}$, we have $Qf=f$.
		\end{enumerate} 
		
		We first prove statement (1). Suppose $\varphi=\xi \cdot F$ is the inner-outer factorization of $\varphi$. By proposition \ref{4.2}, for $g\in L^2(\T, |\varphi|^pdm)$, it holds that $$\E_{\varphi}(g|\eta)=\sum_{k\in \mathbb Z}  \,\la g,\eta^k \ra_{\varphi}\, \eta^k.$$ where $\la \cdot ,\cdot \ra_{\varphi}$ denotes the inner product in $L^2(\T,|\varphi|^p dm)$. In particular, for any fixed analytic polynomial $q$, we have
		\begin{align*}
			\E_{\varphi}(q/\xi | \eta)=\sum_{k\in \mathbb Z}  \la q/\xi, \eta^k \ra_{\varphi} \,\eta^k.
		\end{align*}
		Because $\xi F^{p/2}\perp \eta H^2$, for $k<0$ we have $$\la q/\xi, \eta^k \ra_{\varphi}=\la qF^{p/2}, \xi F^{p/2} \eta^k \ra=0.$$ This results in
		\begin{align}\label{eq4.2}
			\E_{\varphi}(q/\xi | \eta)=\sum_{k=0}^\infty  \la q/\xi, \eta^k \ra_{\varphi} \,\eta^k.
		\end{align} 
		By Bessel's inequality, we have
		\begin{align*}
			\sum_{k=0}^\infty \bigg| \la q/\xi, \eta^k \ra_{\varphi}\bigg|^2 \leq \int_\T |q/\xi|^2 |\varphi|^p \leq \|q\|_{\infty} <\infty,
		\end{align*} 
		and hence the function $$h(z):=\sum_{k=0}^\infty  \la q/\xi, \eta^k \ra_{\varphi} \,z^k \in H^2.$$ Clearly, $\E_{\varphi}(q/\xi | \eta) \in L^\infty$. Combining this with \eqref{eq4.2}, we observe that $$h\circ \eta=\E_{\varphi}(q/\xi | \eta) \in H^2\cap L^\infty =H^\infty.$$ This forces $h\in H^\infty$. As a consequence,
		$$Q(qF)=\varphi\E_{\varphi}(q/\xi |\eta)=\varphi \cdot h\circ \eta\in H^p_{\eta,\varphi}.$$ 
		Because $\{qF: q\in \mathbb C[z]\}$ is dense in $H^p$, $Q$ maps $H^p$ into $H^p_{\eta,\varphi}$. The statement (1) is proved.
		
		To prove the statement (2), it is sufficient to show $$Q(\varphi\cdot q\circ\eta)=\varphi\cdot q\circ\eta$$ for any analytic polynomial $q$. This holds trivially, as the function $q\circ \eta$ is ${\bf \Sigma}(\eta)$-measurable.
	\end{proof}

	We have now proved the sufficiency of Theorem \ref{1.2}. In the remaining part of this section, we proceed to prove the necessity, which can be reduced to the following form.
	
	\begin{theorem}[The necessity part of Theorem \ref{1.2}]\label{4.5}
		Let $1\leq p< \infty$, $p\neq 2$. If $P: H^p\to H^p$ is a contractive projection, then either $P$ is of rank one, or there exists $(\eta,\varphi)\in \X_p$, such that $$P=\varphi\E_{\varphi}(\cdot/\varphi| \eta)\big|_{H^p}.$$ In particular, every nontrivial 1-complemented subspaces of $H^p$ is of the form $H^p_{\eta,\varphi}$ for some $(\eta,\varphi)\in \X_p$.
	\end{theorem} 
	
	The proof of Theorem \ref{4.5} is divided into three steps.
	
	\subsection{Step I: An Application of The Extension Theorems}
	
	In this subsection, we apply the extension theorems established in Section 3 to prove the following weaker form of Theorem \ref{4.5}. 
	
	\begin{theorem}\label{4.6}
		Let $1\leq p< \infty$, $p\neq 2$, and let $P:H^p\to H^p$ be a contractive projection. Let ${\bf \Sigma}^P$ be the $\sigma$-algebra generated by all functions in $$\{f/h: f, h\in \Ran P, h\neq 0\}.$$ Then for any function $\phi\in \Ran P$ with $\|\phi\|_p=1$, it holds that $$Pf = \phi \E_{\phi} (f/\phi |{\bf \Sigma}^P), \quad f\in H^p.$$
	\end{theorem}
	
	To apply the extension theorems, we need to construct a contractive projection that fixes the constants. We do so by using the standard transfer argument.
	
	\begin{definition}\label{4.7}
		For a function $\phi\in H^p$ with $\|\phi\|_p=1$. We define the space $$H^p_\phi:=\{g\in L^p(\T,|\phi|^pdm): g \cdot \phi \in H^p \}.$$
	\end{definition}
	Clearly, $H^p_\phi$ is a closed subspace of $L^p(\T,|\phi|^pdm)$. Moreover, the multiplication operator $$M_\phi: L^p(\T,|\phi|^pdm)\to L^p(\T,dm): \quad g\mapsto \phi \cdot g$$ is an isometric isomorphism that carries $H^p_\phi$ onto $H^p$. It is also easy to verify that $\cal M(H^p_\phi)=H^\infty$
	and the multiplier algebra $\cal M(H^p_\phi)$ satisfies the condition (2) in Theorem \ref{3.3}.


	\begin{proof}[Proof of Theorem \ref{4.6}]
		Let $\phi$ be a unit vector in $\Ran P$. We define 
		\begin{align*}
			P_\phi: H^p_\phi \to H^p_\phi; \quad g\mapsto \frac{P(g\phi)}{\phi}.
		\end{align*}
		Evidently, $P_\phi=M_{1/\phi} P M_\phi |_{H^p_\phi}$, and hence $P_\phi$ is a contractive projection on $H^p_\phi$ that satisfies $P_\phi \one=\one$. If $p\notin 2\mathbb N$, we apply Theorem \ref{3.1}, otherwise we apply Theorem \ref{3.3}. Consequently, as long as $p \neq 2$, the contractive projection $P_\phi$ extends to a conditional expectation $$\E_{\phi}\big(\cdot| {\bf \Sigma}^P_\phi\big): L^p(\mathbb{T},|\phi|^p dm)\to L^p(\mathbb{T},|\phi|^p dm),$$ where ${\bf \Sigma}^P_\phi$ is the $\sigma$-algebra generated by functions of the form $f/\phi$ with $f\in \Ran P$. This shows
		$$M_{1/\phi} P M_\phi \big|_{H^p_\phi}=P_\phi = \E_{\phi}\big(\cdot |{\bf \Sigma}^P_\phi\big)\big|_{H^p_\phi},$$ and hence $$P =M_\phi \E_\phi \big(\cdot | {\bf \Sigma}^P_\phi\big) M_{1/\phi}\big|_{H^p}=\phi \E_{\phi} \big(\cdot /\phi |{\bf \Sigma}^P_\phi\big)\big|_{H^p}.$$ It remains to show ${\bf \Sigma}^P_\phi={\bf \Sigma}^P$ for any $\phi\in \Ran P$ with $\|\phi\|_p=1$. Indeed, suppose $h \in \Ran P$ and $h\neq 0$, then $h/\phi$ is ${\bf \Sigma}^P_\phi$-measurable. Thus for any function $f\in \Ran P$, $f/h=(f/\phi)\cdot (\phi/h)$ is also ${\bf \Sigma}^P_\phi$-measurable. The arbitrariness of $f$ and $h$ implies ${\bf \Sigma}^P \subseteq {\bf \Sigma}^P_\phi$. However, it is clear that ${\bf \Sigma}^P_\phi\subseteq{\bf \Sigma}^P$. Thus ${\bf \Sigma}^P_\phi={\bf \Sigma}^P$. The proof is completed. 
	\end{proof}
	
	As pointed out in the introduction, Theorem \ref{4.6} is sufficient to answer Wojtaszczyk's problem \cite{Wo03}. The details are presented in the following.
	
	\begin{corollary}\label{4.8}
		Let $1 \leq p <\infty$, $p\neq 2$,  and let $P:H^p\to H^p$ be a contractive projection. Then there exists a unique contractive projection $\widetilde{P}$ on $L^p$ that leaves $H^p$ invariant satisfying $$P=\widetilde{P}\big|_{H^p}.$$
	\end{corollary}
	
	\begin{proof}
		From the proof of Theorem \ref{4.6}, we observe that for any $\phi \in \Ran P$, the conditional expectation $\E_\phi(\cdot | {\bf \Sigma}^P)$ on $L^p(\T, |\phi|^p dm)$ always leaves $H^p_{\phi}$ invariant. As a result, if we define
		\[
		\widetilde{P} := M_\phi \, \E_\phi\big(\cdot | {\bf \Sigma}^P\big) \, M_{1/\phi} : L^p \to L^p,
		\]
		then $\widetilde{P}$ is a contractive projection on $L^p$ that leaves $H^p$ invariant satisfying $\widetilde{P}\big|_{H^p} = P$. It remains to show that such an extension $\widetilde{P}$ of $P$ is unique. Assume $\widetilde{P}'$ is a contractive projection on $L^p$ such that $\widetilde{P}'\big|_{H^p}=P$. Pick a unit vector $\phi$ in $\Ran P$. Then both $$\widetilde{P}_\phi:=M_{1/\phi}\widetilde{P} M_\phi\quad {\rm and}\quad \widetilde{P}'_\phi:=M_{1/\phi}\widetilde{P}' M_\phi$$ are contractive projections on $L^p(\T,|\phi|^pdm)$ preserve the constant functions. So both of them are conditional expectations, and therefore positive. In particular, for $f\in L^p(\T,|\phi|^pdm)$ we have $$\widetilde{P}_\phi \ol{f}=\ol{\widetilde{P}_\phi f} \quad {\rm and} \quad  \widetilde{P}'_\phi \ol{f}=\ol{\widetilde{P}'_\phi f}.$$
		Note that $\widetilde{P}_\phi$ coincides with $\widetilde{P}'_\phi$ on $H^p_\phi$, which contains all analytic polynomials. Thus, they are also identical on all trigonometric polynomials. Since trigonometric polynomials are dense in $L^p(\T,|\phi|^pdm)$, we conclude that $\widetilde{P}_\phi=\widetilde{P}'_\phi$, and hence $\widetilde{P}=\widetilde{P}'$.
	\end{proof}

	\subsection{Step II: The Identification of $\eta$}
	
	Let $1\leq p< \infty$, $p\neq 2$, and let $P:H^p\to H^p$ be a contractive projection. Recall that Theorem \ref{4.6} states that, upon fixing any unit vector $\phi\in \Ran P$, the projection $P$ admits the representation $$P=\phi \E_{\phi}(\cdot /\phi| {\bf \Sigma}^P),$$ where ${\bf \Sigma}^P$ is the $\sigma$-algebra generated by all functions in $$\{f/h: f, h\in \Ran P, h\neq 0\}.$$ A key observation is that the conditional expectation $\E_{\phi}(\cdot |{\bf \Sigma}^P)$, as an operator on the weighted space $L^p(\T, |\phi|^p dm)$, must leave $H^p_\phi$ invariant. If $\one \in \Ran P$, then we can chose $\phi=\one$ to conclude that the conditional expectation $\E(\cdot | {\bf \Sigma}^P)$ leaves $H^p$ invariant. In this case, an elegant theorem of A. B. Aleksandrov \cite{Al} can be applied directly.
	
	\begin{Theorem Al}
		The conditional expectation $$\E(\cdot| \mathscr F): L^p(\T) \to L^p(\T); \quad f\mapsto \E(f|\mathscr F),$$ leaves $H^p(\T)$ invariant if and only if $\mathscr F$ is trivial or generated by an inner function vanishing at the origin. 
	\end{Theorem Al}
	
	The main task of this subsection is to show that the $\sigma$-algebra ${\bf \Sigma}^P$ in Theorem \ref{4.6} is always generated by an inner function, even when $\one \notin \Ran P$. 
	\begin{theorem}\label{4.9}
		Let $1\leq p< \infty$, $p\neq 2$. Let $\phi$ be a unit vector in $H^p$ and let $\mathscr F$ be a sub $\sigma$-algebra of ${\bf \Sigma}$. If the conditional expectation $\E_\phi(\cdot | \mathscr F)$ on $L^p(\T,|\phi|^pdm)$ leaves $H^p_\phi$ invariant, then $\mathscr F$ is trivial or $\mathscr F$ is generated by an inner function vanishing at the origin. 
	\end{theorem}
	
	In the following content of this subsection, we always assume $1\leq p<\infty$, $p\neq 2$, $\phi$ is a unit vector in $H^p$, and $\mathscr F$ is a sub $\sigma$-algebra of ${\bf \Sigma}$ such that the conditional expectation $\E_\phi(\cdot |\mathscr F)$ leaves $H^p_\phi$ invariant. To simplify the notation, we denote by $L^p(\mathscr F)$ the space of all $\mathscr F$-measurable functions in $L^p$, and by $H^p(\mathscr F)$ the space of all $\mathscr F$-measurable functions in $H^p$. 
	
	In this subsection and the next, some standard results concerning the Nevanlinna-Smirnov class $N^+$ will be employed. We do not reproduce the definitions or results here. All relevant facts about $N^+$ used in this paper are available in the classical monograph \cite[Chapter II.5]{Gar81}.

	\begin{lemma}\label{4.10}
		The conditional expectation $\E_\phi(\cdot | \mathscr F)$ leaves $H^\infty$ invariant. In particular, $H^\infty(\mathscr F)+\ol{H^\infty(\mathscr F)}$ is $w^*$-dense in $L^\infty(\mathscr F)$.
	\end{lemma}
	\begin{proof}
		Let $\phi = \xi \cdot F$ be the inner-outer factorization of $\phi$. We first claim that $$H^p_\phi \cap L^\infty = H^\infty_{\xi}.$$ The inclusion $H^\infty_\xi \subseteq H^p_\phi \cap L^\infty$ is clear. For the converse, suppose $g\in H^p_{\phi}\cap L^\infty$, then $g=f/\phi$ for some $f\in H^p$. Note that $f/F$ is in the Nevanlinna-Smirnov class $N^+$. As a consequence, $$\xi \cdot g = f/F \in N^+\cap L^\infty=H^\infty.$$ This proves the claim.
		Now assume $f\in H^\infty$. Recall that the multiplier algebra of $H^p_\phi$ is $H^\infty$. From the same reasoning as in the proof of \eqref{eq3.8}, we conclude that $\E_{\phi}(f|\mathscr F)^n \in H^p_\phi$ for any $n\geq 1$. Combining this with the previous claim, we obtain $$\E_{\phi}(f|\mathscr F)^n \in H^\infty_\xi,$$ for any $n\geq 1$. This forces $\E_{\phi}(f | \mathscr F) \in H^\infty$, thus proving the first part of the lemma.
		
		For the second part, we observe that the conditional expectation $\E_\phi(\cdot |\mathscr F)$ maps $L^\infty$ onto $L^\infty (\mathscr F)$, and $$\E_\phi(\cdot |\mathscr F): L^\infty \to L^\infty(\mathscr F),$$ is $w^*$-continuous. Because $H^\infty+\ol{H^\infty}$ is $w^*$-dense in $L^\infty$, the $w^*$-continutiy of $\E_\phi(\cdot |\mathscr F)$ implies $\E_\phi(H^\infty+\ol{H^\infty} \big| \mathscr F)$ is $w^*$-dense in $L^\infty (\mathscr F)$. According to the first part of this lemma and the properties of conditional expectation, we have
		\begin{align*}
			\E_\phi(H^\infty+\ol{H^\infty} \big| \mathscr F)&=\E_\phi(H^\infty\big| \mathscr F)+\E_\phi(\ol{H^\infty} \big| \mathscr F)\\
			&=\E_\phi(H^\infty\big| \mathscr F)+\ol{\E_\phi(H^\infty \big| \mathscr F)}\\
			&= H^\infty(\mathscr F)+ \ol{H^\infty(\mathscr F)}.
		\end{align*}
		Thus $H^\infty(\mathscr F)+\ol{H^\infty(\mathscr F)}$ is $w^*$-dense in $L^\infty(\mathscr F)$.
	\end{proof}
	
	We set 
	$$H_0^2(\mathscr F):= \{f\in H^2(\mathscr F): \int_\T f \,dm =0\}.$$ Clearly, $H^2_0(\mathscr F)$ is a closed subspace of $H^2_0$ and
	\begin{align}\label{eq4.3}
		\ol{H^2_0(\mathscr F)} \perp H^2.
	\end{align}
	\begin{lemma}\label{4.11}
		As Hilbert spaces, we have the orthogonal decomposition $$L^2(\mathscr F)= H^2(\mathscr F)\oplus \ol{H^2_0(\mathscr F)}.$$
	\end{lemma}
	
	\begin{proof}
		Since $\ol{H^2_0(\mathscr F)} \perp H^2(\mathscr F)$, their direct sum $M$ is a closed subspace of $L^2(\mathscr F)$. To prove $M = L^2(\mathscr F)$, it is sufficient to show that for any $f \in L^2(\mathscr F)$, if $f \perp M$, then $f = 0$. Since $M$ contains $H^\infty(\mathscr F) + \ol{H^\infty(\mathscr F)}$, we have
		\[
		\int_\T f \, \ol{g} = 0, \quad \forall g \in H^\infty(\mathscr F) + \ol{H^\infty(\mathscr F)}.
		\]
		By Lemma \ref{4.10}, the $w^*$-density of $H^\infty(\mathscr F) + \ol{H^\infty(\mathscr F)}$ implies that $f = 0$.
	\end{proof}
	
	\begin{proof}[Proof of Theorem \ref{4.9}]
		Since $\E(\cdot |\mathscr F): L^2\to L^2(\mathscr F)$ is the orthogonal projection, it follows from Lemma \ref{4.11} and \eqref{eq4.3} that $\E(f| \mathscr F)\in H^2(\mathscr F)$ for $f\in H^2$. Thus $\E(\cdot |\mathscr F)$ leaves $H^2$ invariant, and the proof is completed by applying Aleksandrov's theorem. 
	\end{proof}

	\subsection{Step III: The Construction of $\varphi$}
	
	Now, we proceed to the final step in the proof of Theorem \ref{4.5}, beginning with an outline of three lemmas for future reference.
	
	\begin{lemma}\label{4.13}
		Let $\eta$ be an inner function with $\eta(0)=0$. Suppose $F\in N^+$ is an outer function. Then $F$ is $\eta$-measurable if and only if $|F|$ is $\eta$-measurable.
	\end{lemma}
	
	The proof of Lemma \ref{4.13} involves some results about the Hilbert transform. For a function $f\in L^1$, let $\cal H f$ be the Hilbert transform of $f$ defined by 
	$$\cal H f(e^{i\theta})= \text{p.v.}\,\frac{1}{2\pi}\int_{-\pi}^{\pi} f(e^{it})\cdot \cot\bigg(\frac{\theta-t}{2}\bigg)dt.$$ Indeed, this singular integral is well defined for almost every $\theta\in[-\pi,\pi]$. We need the following two properties of $\cal H$. 
	\begin{enumerate}
		\item $\cal H: L^2\to L^2$ is bounded, and $$I+i\cal H = 2\cal R - \E,$$ where $\cal R: L^2\to H^2$ is the Riesz projection and $\E f=\int_\T f \,dm$ is the expectation.
		\item If $f_n$ converges to $f$ in $L^1$, then there exists a subsequence of $\{\cal H f_n\}$ that converges to $\cal Hf$ almost everywhere. 
	\end{enumerate}
	We remark that (2) can be deduced by the weak-type (1,1) boundedness of $\cal H$, or by Kolmogorov's theorem, which asserts that $\cal H: L^1\to L^q$ is bounded when $q<1$.
	
	\begin{proof}[Proof of Lemma \ref{4.13}]
		We only need to prove $|F|$ is $\eta$-measurable implies $F$ is $\eta$-measurable. As $F$ is an outer function in $N^+$, $\log|F|\in L^1$ and 
		
		$$F(z)=c\cdot \exp\bigg(  \frac{1}{2\pi} \int_{-\pi}^\pi \frac{e^{it}+z}{e^{it}-z} \log |F(e^{it})| dt \bigg), \quad z\in \D.$$
		where $c$ is a unimodular constant. Taking radial limits on both sides (see \cite[pp.78]{Ho62} and \cite[Theorem 11.5.1]{Ga07}), we obtain 
		$$F=c\cdot \exp\bigg( \log|F| +i \cal H(\log |F|) \bigg).$$ To complete the proof, it is sufficient to show $\cal H(\log |F|)$ is $\eta$-measurable.
		
		Observe that $\E(\cdot|\eta)$ leaves both $H^2$ and $\ol{H^2_0}$ invariant. Thus, as operators on $L^2$, $\E(\cdot|\eta)$ commutes with the Riesz projection $\cal R$. Since $\E(\cdot|\eta)$ also commutes with $\E$, it follows from (1) that $\E(\cdot|\eta)$ commutes with $\cal H$. In particular, $\cal H$ leaves $L^2({\bf \Sigma(\eta)})$ invariant. 
		
		Since $\log |F| \in L^1({\bf \Sigma}(\eta))$, there exists a sequence of functions $\{g_n\}$ in $L^2({\bf \Sigma}(\eta))$, such that $g_n\to \log |F|$ in $L^1$. As a consequence of (2), there exists a subsequence of $\{\cal Hg_n\}$ that converges to $\cal H(\log |F|)$ almost everywhere. Note that every $\cal Hg_n$ is $\eta$-measurable, so is $\cal H(\log|F|)$. The proof is completed.
	\end{proof}
	
	The next lemma is about inner functions. According to Beurling's theorem, every nonzero closed shift-invariant subspace of $H^2$ is of the form $\chi H^2$, where $\chi$ is an inner function. Furthermore, this representation is unique in the sense that if $\chi H^2=\chi'H^2$, then the quotient $\chi'/\chi$ is a unimodular constant. For two inner functions $\chi$ and $\chi'$, we say that $\chi$ {\it divides} $\chi'$ if $\chi'/\chi$ is an inner function, or equivalently, the corresponding shift-invariant subspaces satisfy $\chi'H^2 \subseteq  \chi H^2$. Let $\mathfrak F$ be a family of inner functions, then $$\bigcap_{\mathfrak F\subseteq \chi H^2} \chi H^2$$ is a nonzero closed shift-invariant subspace. By Beurling's theorem, there exists a unique inner function(up to a unimodular constant), denoted by $ \bf{gcd} \,\mathfrak F$, such that $$\bigcap_{\mathfrak F\subseteq \chi H^2} \chi H^2 = {\bf gcd} \,\mathfrak F \cdot H^2.$$
	It is easy to verify that 
	\begin{enumerate}
		\item $\bf{gcd} \,\mathfrak F$ divides any function in $\mathfrak F$.
		\item If $\chi$ is an inner function which divides every function in $\mathfrak{F}$, then $\chi$ divides $\bf{gcd}\, \mathfrak{F}$.
	\end{enumerate}
	We shall call \( \mathbf{gcd} \, \mathfrak{F} \) the \textit{greatest common divisor} of \( \mathfrak{F} \); see \cite[pp.85]{Ho62} for a function-theoretic proof of the existence of \( \mathbf{gcd} \, \mathfrak{F} \).

	Observe ${\bf gcd} \,\mathfrak F \cdot H^2$ is the smallest closed shift-invariant subspace containing $\mathfrak F$. So we have
	\begin{align}\label{eq4.4}
		{\bf gcd} \,\mathfrak F \cdot H^2 = \ol{\sum_{\chi \in \mathfrak F} \chi H^2}^{H^2},
	\end{align}
	where $$\sum_{\chi \in \mathfrak F} \chi H^2:=\bigg\{ \sum_{j=1}^N \chi_j h_j: \chi_j\in \mathfrak{F}, h_j\in H^2, N\in \mathbb N \bigg\}.$$

	\begin{lemma}\label{4.14}
		Suppose $\mathfrak F$ is a family of inner functions, $\eta$ is an inner function with $\eta(0)=0$. Define $$\mathfrak F\circ \eta:= \{\chi\circ\eta: \chi \in \mathfrak{F}\}.$$ Then $${\bf gcd}\,(\mathfrak{F}\circ \eta) = ({\bf gcd}\,\mathfrak{F})\circ \eta.$$
	\end{lemma}
	
	\begin{proof}
		Since $\eta$ is an inner function vanishing at the origin, the composition operator $$C_\eta: H^2\to H^2({\bf \Sigma(\eta)})$$ is an isometric isomorphism that carries $\mathfrak{F}$ onto $\mathfrak{F}\circ{\eta}$. Thus $({\bf gcd}\, \mathfrak{F})\circ \eta$ divides every inner function in $\mathfrak F\circ \eta$, and hence $({\bf gcd}\, \mathfrak{F})\circ \eta$ divides ${\bf gcd}\, (\mathfrak{F}\circ \eta)$. For the converse, considering $C_\eta$ acting on both sides of \eqref{4.4}, we obtain
		$$({\bf gcd}\,\mathfrak{F})\circ \eta\cdot H^2({\bf \Sigma(\eta)})= \ol{  \sum_{\chi\in \mathfrak{F}}  \chi\circ\eta \cdot H^2 (\bf \Sigma(\eta))}^{H^2}.$$ Note that the right-hand side is included in $$\ol{  \sum_{\chi\in \mathfrak{F}}  \chi\circ\eta \cdot H^2 }^{H^2} = \ol{  \sum_{\chi' \in \mathfrak{F}\circ \eta}  \chi' \cdot H^2 }^{H^2} \stackrel{\eqref{eq4.4}}{=} {\bf gcd} \,(\mathfrak F\circ \eta) \cdot H^2.$$ Since $\one \in H^2(\bf \Sigma(\eta))$, we conclude that $$({\bf gcd}\,\mathfrak{F})\circ \eta \in {\bf gcd} \,(\mathfrak F\circ \eta) \cdot H^2.$$ This shows $({\bf gcd} \,\mathfrak F)\circ \eta$ divides ${\bf gcd}\,(\mathfrak{F}\circ \eta)$, as desired.
	\end{proof}
	
	The final lemma is an easy consequence of the properties of conditional expectation.
	\begin{lemma}\label{4.15}
		Let $\E(\cdot|{\bf \Sigma}')$ be the conditional expectation on $L^1(\mu)$ corresponding to the sub $\sigma$-algebra ${\bf \Sigma}'$. If $K$ is a positive function so that $K\in L^1(\mu)$ and $\log K\in L^1(\mu)$, then $K/\E(K |{\bf \Sigma}') \in L^1(\mu)$ and $\log\E(K|{\bf \Sigma}')\in L^1(\mu)$. 
	\end{lemma}
	\begin{proof}
		By the chain rule for the Radon-Nikodym derivative, it holds that $$\frac{1}{\E(K|{\bf \Sigma}')}=\E_K(1/K |{\bf \Sigma}')\in L^1(Kd\mu),$$ where $\E_K(\cdot |{\bf \Sigma}')$ is the corresponding conditional expectation on $L^1(Kdm)$. This implies $K/\E(K|{\bf \Sigma}')\in L^1(d\mu).$ 
		
		Next, we show both $\log_+\E(K|{\bf \Sigma}')$ and $\log_-\E(K|{\bf \Sigma}')$ are integrable, where $$\log_+ x:=\max\{\log x, 0\}, \quad \log_-(x):= \min\{\log x, 0\},\quad x>0.$$
		Observe that
		$$
		\log_+\E(K|{\bf \Sigma}')\leq \log_+(1+\E(K|{\bf \Sigma}'))=\log(1+\E(K|{\bf \Sigma}'))\leq \E(K|{\bf \Sigma}').
		$$
		Thus $\log_+\E(K|{\bf \Sigma}')\in L^1(\mu)$. Since $\log_- x$ is concave, applying Jensen's inequality we have
		$$
		\log_-\E(K|{\bf \Sigma}')\geq \E(\log_-K|{\bf \Sigma}').
		$$
		This ensures $\log_-\E(K|{\bf \Sigma}') \in L^1(\mu)$.
	\end{proof}
	
	Now we are ready to prove Theorem \ref{4.5}.
	
	\begin{proof}[Proof of Theorem \ref{4.5}]
		Suppose $1\leq p< \infty$, $p\neq 2$ and $P: H^p\to H^p$ is a contractive projection. From Theorem \ref{4.6}, for any fixed norm one function $\phi\in \Ran P$, $P$ admits the representation $$Pf = \phi \E_\phi (f/\phi| {\bf \Sigma}^P), \quad f\in H^p,$$ where ${\bf \Sigma}^P$ is the $\sigma$-algebra generated by all functions of the form $f/h$, with $f,h\in \Ran P$, $h\neq 0$. Because $P$ maps $H^p$ into itself, $\E_\phi(\cdot |{\bf \Sigma}^P)$ must leave $H^p_\phi$ invariant. Applying Theorem \ref{4.9}, we conclude that the $\sigma$-algebra ${\bf \Sigma}^P$ is trivial or generated by an inner function vanishing at the origin. If we are in the first case, then by the definition of ${\bf \Sigma}^P$, $P$ is of rank one. So without loss of generality, we assume in the following that ${\bf \Sigma}^P$ is generated by an inner function $\eta$ with $\eta(0)=0$.

		Next, we will show there exists a function $\varphi\in \Ran P$, such that $(\eta,\varphi)\in \X_p$. Once the existence of such $\varphi$ is established, we can choose $\phi=\varphi$ in Theorem \ref{4.6}, which yields $$P=\varphi \E_\varphi(\cdot /\varphi |\eta),$$ and then Theorem \ref{4.5} follows immediately. 
		
		From now on, we will fix a function $\phi\in \Ran P$ with $\|\phi\|_p=1$. With this choice, $P$ admits the representation:
		$$Pf = \phi\E_\phi(f/\phi |\eta), \quad \forall f\in H^p.$$
		In particular, we have 
		\begin{align}\label{eq4.5}
			\Ran P=\{f\in H^p: f/\phi \,\,{\rm is}\, \,{\eta}-{\rm measurable}\}.
		\end{align}
		
		Suppose $\phi= \theta \cdot G $ is the inner-outer factorization of $\phi$. We consider the outer function 
		\begin{align}\label{eq4.6}
			F(z):=\exp\bigg( \frac{1}{2\pi}\int_{-\pi}^{\pi} \frac{e^{it}+z}{e^{it}-z} \cdot \log  \bigg| \frac{|\phi|}{(\E(|\phi|^p| \eta))^{1/p}}\bigg| dt \,\bigg).
		\end{align}
		In the virtual of Lemma \ref{4.15}, $F$ is well defined and $F\in H^p$, with boundary value 
		$$|F|^p= \frac{|\phi|^p}{\E(|\phi|^p| \eta)}.$$ 
		
		On the other hand, we consider the outer function $H$ determined by $|H|=|\E(\theta|\eta)|$. By Lemma \ref{4.13}, $H$ is $\eta$-measurable. Next we define
		$$\mathfrak{F}:= \{\chi: \chi \,\,{\rm is\,\,inner},\,\chi/\theta  \,\,{\rm is}\, \,{\eta}-{\rm measurable} \}.$$ Note that $\mathfrak{F}$ is nonempty as $\theta \in \mathfrak{F}$. For any $\chi \in \mathfrak{F}$, since $\chi/\theta$ is $\eta$-measurable, applying the averaging identity we deduce that $\E(\chi|\eta)=\chi /\theta \cdot \E(\theta|\eta)$. As a consequence,
		\begin{align}\label{eq4.7}
			\theta\cdot\E(\chi|\eta)=\chi \cdot \E(\theta|\eta),\quad \forall \chi\in \mathfrak{F}.
		\end{align}
		Note that \eqref{eq4.7} implies for any $\chi \in \mathfrak{F}$, the outer part of $\E(\chi|\eta)$ is $H$. Thus, the quotient $\E(\chi|\eta)/H$ is an $\eta$-measurable inner function. This means, for any $\chi\in \mathfrak{F}$, there exists an inner function $\widetilde{\chi}$ depending on $\chi$, such that the inner-outer factorization of $\E(\chi |\eta)$ reads as
		\begin{align}\label{eq4.8}
			\E(\chi |\eta)= \widetilde{\chi} \circ \eta \cdot H.
		\end{align}
		Substituting \eqref{eq4.8} into \eqref{eq4.7}, we obtain 
		\begin{align}\label{eq4.9}
			\theta \cdot \widetilde{\chi} \circ \eta = \chi \cdot \widetilde{\theta} \circ \eta, \quad \forall \chi \in \mathfrak{F}.
		\end{align}
		Set $$\widetilde{\mathfrak{F}}:=\{\widetilde{\chi}: \chi\in \mathfrak{F}\}.$$
		Considering the greatest common divisors on both side of \eqref{eq4.9}, and applying Lemma \ref{4.14}, we deduce that
		\begin{align*}
			\theta\cdot ({\bf gcd}\, \widetilde{\mathfrak{F}})\circ \eta=\theta\cdot {\bf gcd}\, (\widetilde{\mathfrak{F}}\circ \eta) = \widetilde{\theta}\circ\eta \cdot {\bf gcd}\,\mathfrak{F}.
		\end{align*}
		In particular, the function $$\frac{{\bf gcd}\,\mathfrak{F}}{\theta }= \frac{({\bf gcd}\, \widetilde{\mathfrak{F}})\circ \eta}{\widetilde{\theta}\circ\eta} $$ is $\eta$-measurable. This shows ${\bf gcd}\,\mathfrak{F}\in \mathfrak{F}$. 
		
		Now we put $\xi:= {\bf gcd}\,\mathfrak{F},$ and $$\varphi:=\xi \cdot F,$$ where $F$ is defined in \eqref{eq4.6}. In fact, this $\varphi$ is what we desired.
		
		Let us check first that $\varphi \in \Ran P$. Indeed, because $\xi\in \mathfrak F$, by our definition of $\mathfrak F$, $\xi/\theta$ is $\eta$-measurable. On the other hand, note that the quotient of outer functions $F/G$ is an outer function in $N^+$, and $$|F/G|=\E(|\phi|^p|\eta)^{-1/p}$$ is $\eta$-measurable. By Lemma \ref{4.13}, $F/G$ is $\eta$-measurable. We deduce that $$\frac{\varphi}{\phi}=\frac{\xi}{\theta} \cdot \frac{F}{G}$$ is $\eta$-measurable. Combining this with \eqref{eq4.5}, we obtain $\varphi\in \Ran P$. 
		
		It remains to prove $(\eta,\varphi)\in \X_p$. It is clear that $(\eta,\varphi)$ satisfies the conditions (1) and (2) in Definition \ref{4.1}. We only need to show $\xi\cdot F^{p/2}\perp \eta H^2.$ Let $q$ be an analytic polynomial. Since $\E_\varphi(q/\xi |\eta)\in L^\infty({\bf \Sigma(\eta)})$, there exists a function $h\in L^\infty$ such that $$\E_\varphi(q/\xi |\eta)=h\circ \eta.$$ 
		We claim that $h\in H^\infty$. Indeed, by Theorem \ref{4.6}, $$ \varphi\cdot h\circ \eta=\varphi \E_{\varphi}(q/\xi |\eta)=P(qF) \in \Ran P.$$ If we write 
		\begin{align}\label{eq4.10}
			\varphi\cdot h\circ \eta = \alpha \cdot K,
		\end{align}
		where $\alpha$ is inner and $K$ is outer, then it follows from \eqref{eq4.5} that $$\frac{\alpha \cdot K}{\theta\cdot G}$$ is $\eta$-measurable. Combining this with Lemma \ref{4.13}, we deduce that $K/G$ is $\eta$-measurable, and hence $\alpha/\theta$ is $\eta$-measurable. This means $\alpha \in \mathfrak F.$ Recall that $\xi ={\bf gcd}\,\mathfrak{F}$, so $\xi$ divides $\alpha$. Substituting $\varphi=\xi \cdot F$ into \eqref{eq4.10}, we see 
		$$h\circ \eta = \frac{\alpha}{\xi}\frac{K}{F}.$$ Because $\xi$ divides $\alpha$, we conclude that $$h\circ \eta \in N^+\cap L^\infty = H^\infty.$$ This forces $h\in H^\infty$. Our claim follows.
		
		Our construction of $F$ gives that 
		$$\E(|\varphi|^p |\eta)=\one,$$ 
		and so $\{\eta^k: k\in \mathbb Z\}$ is an orthonormal system in $L^2(|\varphi|^p dm)$. Together with the previous claim, we have
		\begin{align*}
			\la \eta \cdot q F^{p/2}, \xi F^{p/2}\ra&=\int_\T \eta \cdot q/\xi \, |\varphi|^p dm\\
			&=\int_{\T}  \eta \cdot \E_{\varphi}(q/\xi |\eta)\, |\varphi|^p dm \\
			&= \int_\T \eta \cdot h\circ \eta \, |\varphi|^p dm \\
			&=0.
		\end{align*}
		Since $F^{p/2}$ is an outer function in $H^2$, $\{q\cdot F^{p/2}: q\in \mathbb C[z]\}$ is dense in $H^2$. This implies $\xi \cdot F^{p/2} \perp \eta H^2$. The proof of Theorem \ref{4.5} is completed.
	\end{proof}

	\section{Appendix}
	
	As mentioned in the introduction, we now present the details of how to deduce the Lancien-Randrianantoanina-Ricard theorem and its codimensional version (Corollary \ref{1.5}) from Theorem \ref{1.6}. In fact, we obtain the following more general result.

	\begin{theorem}\label{5.1}
		Let $1\leq p<\infty, p\neq 2$, and let $(\Omega, \bf \Sigma, \mu)$ be a complete probability space. Assume $X$ is a closed subspace of $L^p(\mu)$ satisfying:
		\begin{enumerate}
			\item For any non-zero function $f\in X$, $\mu\big(\{f=0\}\big)=0$.
			\item The intersection of $\mathcal{M}(X)+\overline{\mathcal{M}(X)}$ and the closed unit ball of $L^\infty$ is w*-dense in the closed unit ball of $L^\infty(\mu)$.
		\end{enumerate}
		Then $X$ admits no nontrivial 1-complemented subspace of finite dimension. If we further assume $(\Omega, \bf \Sigma, \mu)$ is atomless, then $X$ admits no proper 1-complemented subspace of finite codimension.
	\end{theorem}
	
	\begin{proof}
		
		Suppose that $P$ is a contractive projection on $X$. Let $\varphi$ in $\Ran P$ be a fixed unit vector. Following the standard transfer argument, we deduce that 
		$$P_\varphi : X_\varphi \to X_\varphi; \quad f\mapsto \frac{P(\varphi f)}{\varphi}$$ is a contractive projection, where $$X_\varphi:= \{f\in L^p(|\varphi|^pd\mu): \varphi f\in X\}.$$ Applying the extension theorems (note that $\cal M(X_\varphi)=\cal M(X)$), we see that $P_\varphi$ extends to a conditional expectation $\E_\varphi(\cdot | {\bf \Sigma}_\varphi^P)$ on $L^p(|\varphi|^pd\mu)$, where ${\bf \Sigma}_\varphi^P$ is the sub $\sigma$-algebra generated by functions in $\Ran P/\varphi$. Moreover, the assumption (2) ensures that $X_\varphi+ \ol{X_\varphi}$ is dense in $L^p(|\varphi|^pd\mu)$.
		
		If $P$ is of finite rank, then so is $P_\varphi$. Since $\mathbb{E}_\varphi(\cdot|{\bf \Sigma}_\varphi^P)$ is a positive operator, the density of $X_\varphi+ \ol{X_\varphi}$ in $L^p(|\varphi|^pd\mu)$ implies $\mathbb{E}_\varphi(\cdot|{\bf \Sigma}_\varphi^P)$ is also of finite rank. Therefore, $L^p(\Omega, {\bf \Sigma}_\varphi^P, |\varphi|^p d\mu)$ is of finite dimension, and hence ${\bf \Sigma}_\varphi^P$ is generated by a finite measurable partition $$\Omega=\bigsqcup^n_{k=1} F_k,$$ where $F_k\in {\bf \Sigma}_\varphi^P$ are disjoint non null sets. Therefore, any ${\bf \Sigma}_\varphi^P$-measurable function $f$ is of the form 
		\begin{align*} 
			f=\sum_{k=1}^n c_k \one_{F_k},\quad c_k\in \C.
		\end{align*}
		The above identity holds for $f=g/\varphi$ with any $g\in \Ran P$. In particular, $g-c_1\varphi=0$ on $F_1$, and by assumption (1), $g/\varphi=c_1$ almost everywhere. Hence, for any $g\in \Ran P$, we have that $g$ is a complex constant multiple of $\varphi$. This shows $P$ is of rank one.
		
		Now suppose that ${\rm codim}\, \Ran P<\infty$. By the same reasoning, we deduce that the space $L^p(\Omega,{\bf \Sigma}^P_\varphi, |\varphi|^p d\mu)$ has finite codimension in $L^p(\Omega,{\bf \Sigma}, |\varphi|^pd\mu)$. If the probability space $(\Omega,{\bf \Sigma}, |\varphi|^pd\mu)$ is assumed to be atomless, then it is easy to show that $${\rm codim}\,L^p(\Omega,{\bf \Sigma}^P_\varphi, |\varphi|^p d\mu)<\infty$$ implies ${\bf\Sigma}^P_\varphi={\bf \Sigma}$. Thus $P$ is the identity. The proof is completed.
	\end{proof}

	\begin{center}
		{\bf{Acknowledgements}}
	\end{center}
	We are deeply grateful to Professor Kristian Seip for his valuable discussions and helpful suggestions. We would also like to express our sincere gratitude to Professor Yanqi Qiu for providing a simplification of the original proof of Lemma \ref{3.2}.
	\vspace{0.2cm}
	
	\noindent {\bf Funding~} 
	This work was supported by the National Key R\&D Program of China (2024YFA1013400) and the National Natural Science Foundation of China (Grant No. 12231005). Dilong Li was supported by the  National Natural Science Foundation of China (Grant No. 124B2005).


\end{document}